\documentclass[12pt]{amsart}
\usepackage[nobysame,abbrev,alphabetic]{amsrefs}
\usepackage{amssymb}
\usepackage[all]{xy}

\DeclareMathOperator{\Br}{Br}

\DeclareMathOperator{\Cor}{cor}

\DeclareMathOperator{\Dec}{Dec}
\DeclareMathOperator{\Gal}{Gal}
\DeclareMathOperator{\Hom}{Hom}

\DeclareMathOperator{\Img}{Im}

\DeclareMathOperator{\Ker}{Ker}

\DeclareMathOperator{\res}{res}

\begin{document}

\newtheorem{thm}{Theorem}[section]
\newtheorem{cor}[thm]{Corollary}
\newtheorem{lem}[thm]{Lemma}
\newtheorem{prop}[thm]{Proposition}
\newtheorem{defin}[thm]{Definition}
\newtheorem{exam}[thm]{Example}
\newtheorem{examples}[thm]{Examples}
\newtheorem{rem}[thm]{Remark}
\newtheorem*{thmA}{Theorem A}
\newtheorem*{thmA'}{Theorem A'}
\newtheorem*{main thm}{Main Theorem}
\swapnumbers
\newtheorem{rems}[thm]{Remarks}
\newtheorem*{acknowledgment}{Acknowledgment}
\numberwithin{equation}{section}

\newcommand{\dec}{\mathrm{dec}}
\newcommand{\dirlim}{\varinjlim}
\newcommand{\discup}{\mathbin{\mathaccent\cdot\cup}}
\newcommand{\nek}{,\ldots,}
\newcommand{\inv}{^{-1}}
\newcommand{\Inv}{\mathrm{inv}}
\newcommand{\isom}{\cong}
\newcommand{\Massey}{\mathrm{Massey}}
\newcommand{\ndiv}{\hbox{$\,\not|\,$}}
\newcommand{\pr}{\mathrm{pr}}
\newcommand{\tensor}{\otimes}
\newcommand{\U}{\mathrm{U}}
\newcommand{\alp}{\alpha}
\newcommand{\gam}{\gamma}
\newcommand{\del}{\delta}
\newcommand{\eps}{\epsilon}
\newcommand{\lam}{\lambda}
\newcommand{\Lam}{\Lambda}
\newcommand{\sig}{\sigma}
\newcommand{\dbF}{\mathbb{F}}
\newcommand{\dbQ}{\mathbb{Q}}
\newcommand{\dbR}{\mathbb{R}}
\newcommand{\dbU}{\mathbb{U}}
\newcommand{\dbZ}{\mathbb{Z}}

\title{Vanishing of Massey products and Brauer groups}

\author{Ido Efrat and Eliyahu Matzri}

\address{Department of Mathematics \\
         Ben-Gurion University of the Negev\\
         Be'er-Sheva 84105 \\
         Israel}
\email{efrat@math.bgu.ac.il, elimatzri@gmail.com}
\thanks{$^1$The authors were supported by the Israel Science Foundation (grant No.\ 152/13).
The second author was also partially supported by the Kreitman foundation.}



\begin{abstract}
Let $p$ be a prime number and $F$ a field containing a root of unity of order $p$.
We relate recent results on vanishing of triple Massey products in the mod-$p$ Galois cohomology of $F$,
due to Hopkins, Wickelgren, Min\'a\v c, and T\^an, to classical results in the theory of central simple algebras.
For global fields, we prove a stronger form of the vanishing property.
\end{abstract}

\maketitle

\section{Introduction}
We fix a prime number $p$.
Let $F$ be a field, which will always be assumed to contain a root of unity of order $p$.
Let $G_F$ be the absolute Galois group of $F$.
Recent works by Hopkins, Wickelgren, Min\'a\v c, and T\^an revealed a remarkable new property of the mod-$p$ Galois cohomology groups $H^i(G_F,\dbZ/p)$, $i=1,2$, related to triple Massey products.
This property, which they proved in several important cases, puts new restrictions on the possible group-theoretic structure of  maximal pro-$p$ Galois groups of fields, and in particular, of absolute Galois groups.
In particular, in \cite{MinacTan13} Min\'a\v c and T\^an use this method to give new examples of pro-$2$ groups which cannot occur as absolute Galois groups of fields.

More specifically, for an arbitrary profinite group $G$ which acts trivially on $\dbZ/p$,
let $H^i(G)=H^i(G,\dbZ/p)$.
We recall that the triple Massey product is a multi-valued map $\langle\cdot,\cdot,\cdot\rangle\colon H^1(G)^3\to H^2(G)$  (see \S\ref{section on Massey products} for its precise definition).
We consider the following cohomological condition:

{\it If $\chi_1,\chi_2,\chi_3\in H^1(G)$ and $\langle\chi_1,\chi_2,\chi_3\rangle\subseteq H^2(G)$ is nonempty, then it contains zero.}

Following Min\'a\v c and T\^an, we call this condition the {\sl vanishing triple Massey product property} for $G$.

When $G=G_F$ for a field $F$ as above,  this property can be rephrased in a more basic Galois-theoretic language,
in terms of the groups $\dbU_n(\dbF_p)$ of unipotent upper-triangular $n\times n$ matrices over $\dbF_p$.
Namely, $\chi_1,\chi_2,\chi_3$ are the Kummer characters corresponding the elements $a_1,a_2,a_3$ of $F^\times$, which we assume for simplicity to be $\dbF_p$-linearly independent in $F^\times/(F^\times)^p$.
Then (see Corollary \ref{Massey products and Un}):
\begin{itemize}
\item
 $\langle\chi_1,\chi_2,\chi_3\rangle\neq\emptyset$ if and only if each of $F(a_1^{1/p},a_2^{1/p})$ and $F(a_2^{1/p},a_3^{1/p})$ embeds inside a Galois extension of $F$ with Galois group $\dbU_3(\dbF_p)$;
\item
$0\in\langle\chi_1,\chi_2,\chi_3\rangle$ if and only if $F(a_1^{1/p},a_2^{1/p},a_3^{1/p})$ embeds inside a Galois extension of $F$ with Galois group $\dbU_4(\dbF_p)$.
\end{itemize}

The vanishing triple Massey product condition was shown to hold for $G=G_F$ in the following situations:
\begin{enumerate}
\item[(i)]
$p=2$ and $F$ is a global field \cite{HopkinsWickelgren15};
\item[(ii)]
$p=2$ and $F$ is arbitrary \cite{MinacTan13};
\item[(iii)]
$p$ is arbitrary and $F$ is a global field \cite{MinacTan14}.
\end{enumerate}
Moreover, Min\'a\v c and T\^an conjecture that $G_F$ satisfies the vanishing triple Massey product property for every field $F$
containing a root of unity of order $p$ \cite{MinacTan15}.
If true, this would give new kinds of examples of pro-$p$ groups which are not realizable as absolute Galois groups for arbitrary primes $p$,
along the lines of \cite{MinacTan13} (where this is done for $p=2$).
In view of the interpretations of triple Massey products in terms of $\dbU_3(\dbF_p)$- and $\dbU_4(\dbF_p)$-Galois extension,
one also obtains an ``automatic realization" principle in the above cases, and conjecturally always.

In this note we relate these recent developments to classical results in the theory of central simple algebras and Brauer groups.
We investigate another cohomological property of $G$, which we call the \textsl{cup product--restriction property} for characters $\chi_1\nek\chi_r\in H^1(G)$.
This property for $r=2$ implies the vanishing triple Massey product property, but unlike the latter property, it does not involve external cohomological operations.
Now when $G=G_F$, we prove this cup product--restriction property for global fields and arbitrary $r$, using the Albert--Brauer--Hasse--Noether theorem and an injectivity theorem for $H^1$ due to Artin and Tate.
In the case where $p=2$ and $r=2$ the cup-product--restriction property for $G_F$ was proved by Tignol \cite{Tignol81a}*{Cor.\ 2.8}, and is an easy consequence of  a refinement, also due to Tignol  \cite{Tignol79}*{Th.\ 1},  of a result of Albert on the decomposition of central simple algebras as a tensor product of two quaternion algebras.

Specifically, let $G$ be an arbitrary profinite group, let $\chi_1\nek\chi_r\in H^1(G)=\Hom(G,\dbZ/p)$, and set $K=\bigcap_{i=1}^r\Ker(\chi_i)$.
We define a multi-linear map $\Lam_{(\chi_i)}\colon H^1(G)^r\to H^2(G)$ by $\Lam_{(\chi_i)}(\varphi_1\nek\varphi_r)=\sum_{i=1}^r\chi_i\cup\varphi_i$.
It lifts to a homomorphism $\Lam_{(\chi_i)}\colon H^1(G)^{\tensor r}\to H^2(G)$.
We say that the \textsl{cup product--restriction property} holds for $\chi_1\nek\chi_r$ if the sequence
\begin{equation}
\label{cup product-restriction}
H^1(G)^{\tensor r}\xrightarrow{\Lam_{(\chi_i)}}H^2(G)\xrightarrow{\res_K}H^2(K)
\end{equation}
is exact.
Note that (\ref{cup product-restriction}) is always a complex.
Further,  its exactness depends only on $K$ (but not on the choice of $\chi_1\nek\chi_r$; see Proposition \ref{dependence on K}).

Now when $\chi_1,\chi_2,\chi_3\in H^1(G)$, the cup product--restriction property for $\chi_1,\chi_3\in H^1(G)$ implies the vanishing triple Massey product
property for  $\langle\chi_1,\chi_2,\chi_3\rangle$ (Proposition \ref{cohomological lemma}).

\begin{main thm}
\label{thm on relative Brauer group}
Let $a_1\nek a_r\in F^\times$.
Suppose that one of the following conditions holds:
\begin{enumerate}
\item[(1)]
$F$ is a global field;
\item[(2)]
$p=2$ and $r=2$.
\end{enumerate}
Then the cup product--restriction property holds for the Kummer elements $(a_1)\nek (a_r)$ in $H^1(F)$.
\end{main thm}

In particular, by restricting to $r=2$, this gives alternative algebra-theoretic proofs of the  results of Hopkins, Wickelgren, Min\'a\v c, and T\^an on the vanishing of triple Massey products, and relates these works to the above-mentioned classical results on Brauer groups.
As remarked above, the case (2) was earlier proved in \cite{Tignol81a}*{Cor.\ 2.8}, and is brought here in order to relate the respective results on Massey products to their algebra-theoretic counterparts (see \S5).
The cup-product--restriction property for $G=G_F$ is also closely related to the subgroup $\Dec(L/F)$ of ${}_p\Br(F)$, introduced in \cite{Tignol81a}, \cite{Tignol81b} where $L$ is the fixed field of $K$ in the separable closure of $F$.

In the case $p=2$ the cup-product--restriction property for absolute Galois groups of fields is essentially the property $P_2$ studied in \cite{Tignol81a} and \cite{ElmanLamTignolWadsworth83}.
In particular, in this situation, case (1) of the Main Theorem was earlier proved in \cite{ElmanLamTignolWadsworth83}*{Cor.\ 3.18}.

On the other hand, constructions of Tignol \cite{Tignol87} and McKinnie \cite{McKinnie11} show that, for $p$ odd, there exist fields $F$ for which the cup product--restriction property with $r=2$ does not hold (Example \ref{McKinnies example}).
This shows that some assumptions on $F$, as in the Main Theorem, are needed.

We thank the referees of this paper for very valuable comments and in particular for drawing our attention to several important references, which we included in the reference list.

\section{Preliminaries in Galois cohomology}
Let $F$ be again a field containing a root of unity of order $p$.
We abbreviate $H^i(F)=H^i(G_F,\dbZ/p)$.
We fix an isomorphism between the group $\mu_p$ of the $p$-th roots of unity and $\dbZ/p$.
This isomorphism induces the Kummer isomorphism $H^1(F)\isom F^\times/(F^\times)^p$.
Given $a\in F^\times$ let $(a)\in H^1(F)$ be the corresponding Kummer element.

Let $\Br(F)$ be the Brauer group of $F$ and let ${}_p\Br(F)$ be its subgroup consisting of all elements with exponent dividing $p$.
Given a field extension $L/F$ we write $\Br(L/F)$ for the kernel of the restriction map $\Br(F)\to\Br(L)$.
The isomorphism $\mu_p\isom\dbZ/p$ also induces in a standard way an isomorphism  $H^2(F)\isom{}_p\Br(F)$.
For $a,b\in F^\times$, let $(a,b)_F$ be the corresponding symbol $F$-algebra of degree $p$.
The cup product $(a)\cup(b)$ in $H^2(F)$ then corresponds to the similarity class  $[(a,b)_F]$ in ${}_p\Br(F)$.

\section{The cup product--restriction property}
\label{comments on cup res property}

We first show that the cup product--restriction property depends only on the subgroup $K=\bigcap_{i=1}^r\Ker(\chi_i)$, but not on the choice of $\chi_1\nek\chi_r$.

\begin{prop}
\label{dependence on K}
Let $G$ be a profinite group and let $K$ be an open subgroup of $G$.
Consider $\chi_1\nek \chi_r,\mu_1\nek\mu_s\in H^1(G)$ such that  $K=\bigcap_{i=1}^r\Ker(\chi_i)=\bigcap_{j=1}^s\Ker(\mu_j)$.
Then the cup product--restriction property holds for $\chi_1\nek \chi_r$ if and only if it holds for $\mu_1\nek\mu_s$.
\end{prop}
\begin{proof}
There is a perfect pairing
\[
G/G^p[G,G]\times H^1(G)\to\dbZ/p, \quad (\bar g,\varphi)\mapsto \varphi(g)
\]
\cite{EfratMinac11}*{Cor.\ 2.2}.
It induces a perfect pairing $G/K\times\langle\chi_1\nek\chi_r\rangle\to\dbZ/p$, and similarly for the $\mu_j$.
Therefore $\langle\chi_1\nek \chi_r\rangle=\langle\mu_1\nek\mu_s\rangle$.
It follows that
\[
\chi_1\cup H^1(G)+\cdots+\chi_r\cup H^1(G)=\mu_1\cup H^1(G)+\cdots+\mu_s\cup H^1(G),
\]
i.e.,  the homomorphisms
\[
\Lam_{(\chi_i)}\colon H^1(G)^{\tensor r}\to H^2(G), \quad \Lam_{(\mu_j)}\colon H^1(G)^{\tensor s}\to H^2(G)
\]
have the same image, and the assertion follows.
\end{proof}

Consequently, for every open subgroup $K$ of $G$ such that $G/K$ is an elementary abelian $p$-group,
we may define the \textsl{cup product--restriction property for $K$} to be the cup product--restriction property
 for $\chi_1\nek\chi_r$, where $\chi_1\nek\chi_r$ is any list of elements in $H^1(G)$ such that $K=\bigcap_{i=1}^r\Ker(\chi_i)$.

\begin{exam}
\label{example r=1}
\rm
Let $\chi=\chi_1\in H^1(G)$, so $K=\Ker(\chi)$.
The cup product--restriction property for $\chi$ means that the sequence
\begin{equation}
\label{exact sequence for r=1}
H^1(G)\xrightarrow{\chi\cup} H^2(G)\xrightarrow{\res} H^2(K),
\end{equation}
is exact.
This is trivial when $\chi=0$, so we assume that $\chi\neq0$ and therefore $(G:K)=p$.

When $p=2$ (\ref{exact sequence for r=1}) is always exact, and is in fact a segment of the infinite Arason exact sequence \cite{Arason75}*{Satz 4.5}:
\[
\cdots\xrightarrow{\res}H^i(K)\xrightarrow{\Cor}H^i(G)\xrightarrow{\chi\cup}H^{i+1}(G)
\xrightarrow{\res}H^{i+1}(K)\xrightarrow{\Cor}\cdots\ .
\]

When $p$ is an arbitrary prime and $F$ is a field (containing as always a fixed root of unity of order $p$), we may write  $\chi=(a)$ for some $a\in F^\times$.
Then $L=F(a^{1/p})$ is the $\dbZ/p$-Galois extension corresponding to $K$.
There is an isomorphism
\[
F^\times/N_{L/F}L^\times\xrightarrow{\sim}\Br(L/F), \quad xN_{L/F}L^\times\mapsto [(a,x)_F]
\]
\cite{Draxl83}*{p.\ 73, Th.\ 1}.
It follows that (\ref{exact sequence for r=1}) is exact, i.e., the cup product--restriction property for $\chi$ holds.
This is again a part of a more general fact:
Based on results of Voevodsky  \cite{Voevodsky03}*{\S5}, it was shown by Lemire, Min\'a\v c and Swallow \cite{LemireMinacSwallow07}*{Th.\ 6} that for every $i\geq1$ the following sequence is exact:
\[
H^i(L)\xrightarrow{\Cor}H^i(F)
\xrightarrow{\chi\cup}H^{i+1}(F)\xrightarrow{\res} H^{i+1}(L),
\]
\end{exam}

\begin{rem}
\label{Dec}
\rm
Suppose that  $F$ is a field containing a root of unity of order $p$, and $L$ is a Galois extension of $F$ with $\Gal(L/F)$ an elementary abelian $p$-group.
As in \cite{Tignol81a}, \cite{Tignol81b} let $\Dec(L/F)$ be the subgroup of $\Br(L/F)$ generated by all subgroups $\Br(L'/F)$, where $L'$ ranges over all cyclic $p$-extensions of $F$ contained in $L$.
Then the cup product--restriction property holds for the subgroup $G_L$ of $G_F$ if and only if $\Dec(L/F)={}_p\Br(L/F)$.
\end{rem}

\begin{exam}
\label{McKinnies example}
\rm
For $p$ odd there are constructions due to Tignol \cite{Tignol87}*{Th.\ 1, Rem.\ 1.3(a)} and McKinnie \cite{McKinnie11} (see also \cite{Saltman79} and \cite{Rowen82} for related works), of division algebras $D$ over a field $F$ which contains a root of unity of order $p$, such that
\begin{enumerate}
\item[(a)]
$D$ splits in $L=F(a_1^{1/p},a_2^{1/p})$ for certain $a_1,a_2\in F^\times$;
\item[(b)]
$[D]\not\in\Dec(L/F)$.
\end{enumerate}
In view of Remark \ref{Dec}, this means that the cup product--restriction property does not hold for the subgroup $K=G_L$ of $G=G_F$.
\end{exam}

\section{Global fields}
Throughout this section we assume that $F$ is a global field containing a root of unity of order $p$.
For every place $v$ on $F$ we write $F_v$ for the completion of $F$ relative to $v$, and denote the canonical extension of $v$ to $F_v$ also by $v$.
There is a canonical monomorphism $\Inv_v\colon \Br(F_v)\to\dbQ/\dbZ$ which is an isomorphism for $v$ non-archimedean.
Restricting to ${}_p\Br(F)$, we obtain a monomorphism $\Inv_v\colon H^2(F_v)\to\frac1p\dbZ/\dbZ$.
It is an isomorphism unless $v$ is archimedean and $p\neq2$ (and in the latter case $H^2(F_v)=0$).
Given a finite extension $E$ of $F_v$, there is a commutative square
\[
\xymatrix{
\Br(E)\ar[r]^{\Inv_u} &\dbQ/\dbZ \\
\Br(F_v)\ar[r]^{\Inv_v}\ar[u]^{\res_E} &\dbQ/\dbZ\ar[u]_{[E:F_v]},
}
\]
where the map on the right means multiplication by the degree $[E:F_v]$ \cite{SerreLocalFields}*{Ch.\ XIII, \S3, Prop.\ 7}.
Consequently, if $p|[E:F_v]$, then $\res_E\colon H^2(F_v)\to H^2(E)$ is the zero map.

We recall that, by classical results of Albert, Brauer, Hasse and Noether, the following sequence is exact:
\[
0\to\Br(F)\xrightarrow{\res}\bigoplus_v\Br(F_v)\xrightarrow{\sum_v\Inv_v}\dbQ/\dbZ\to 0.
\]
It gives rise to an exact sequence
\begin{equation}
\label{LGP}
0\to H^2(F)\xrightarrow{\res}\bigoplus_v H^2(F_v)\xrightarrow{\sum_v\Inv_v}\tfrac1p\dbZ/\dbZ.
\end{equation}

\begin{lem}
\label{v0}
Let  $S$ be a finite set of places on $F$ and let $a_1\nek a_r\in F^\times$.
Suppose that $a_1\not\in (F^\times)^p$.
Then there exists a place $v_0$ on $F$ such that $v_0\not\in S$ and
$a_1,a'_2\nek a'_r\not\in (F_{v_0}^\times)^p$, where for each $2\leq i\leq r$ either $a'_i=a_i$ or $a'_i=a_1a_i$.
\end{lem}
\begin{proof}
The restriction map $H^1(F)\to\prod_{v\not\in S}H^1(F_v)$ is injective \cite{ArtinTate}*{Ch.\ IX, Th.\ 1}.
Hence there is a place $v_0\not\in S$ with $(a_1)\neq0\in H^1(F_{v_0})$.
Then for every $i$ with $2\leq i\leq r$ we have $(a_i)\neq0\in H^1(F_{v_0})$ or $(a_1a_i)=(a_1)+(a_i)\neq0\in H^1(F_{v_0})$, and we can choose $a'_i$ accordingly.
\end{proof}

\medskip

\begin{proof}[Proof of Case (1) of the Main Theorem]
Let $a_1\nek a_r\in F^\times$.
If $a_1\nek a_r\in (F^\times)^p$,  then the cup product--restriction property for $(a_1)\nek(a_r)\in H^1(F)$ is trivial.
We may therefore assume that $a_1\not\in (F^\times)^p$.
Let $L=F(a_1^{1/p}\nek a_r^{1/p})$,
and consider $\alp\in H^2(F)$ with $\res_L(\alp)=0$ in $H^2(L)$.
We have to show that $\alp\in (a_1)\cup H^1(F)+\cdots+ (a_r)\cup H^1(F)$.

To this end let $S$ be the set of all places $v$ on $F$ such that $\alp_{F_v}\neq0$, where $\alp_{F_v}$ denotes the restriction of $\alp$ to $H^2(F_v)$.
By (\ref{LGP}), $S$ is finite.
Let $v_0$ and $a'_2\nek a'_r$ be as in Lemma \ref{v0}.
In particular,  $\alp_{F_{v_0}}=0$.
Now $L=F(a_1^{1/p},(a'_2)^{1/p}\nek (a'_r)^{1/p})$ and
\[
\begin{split}
&(a_1)\cup H^1(F)+(a_2)\cup H^1(F)+\cdots+(a_r)\cup H^1(F)\\
=&(a_1)\cup H^1(F)+(a'_2)\cup H^1(F)+\cdots+(a'_r)\cup H^1(F).
\end{split}
\]
We may therefore replace $a_i$ by $a'_i$, $i=2\nek r$, to assume without loss of generality that $a_1\nek a_r\not\in (F_{v_0}^\times)^p$.

Next for $1\leq i\leq r$ let
\[
S_i=\{v\ |\  a_i^{1/p}\not\in F_v\}.
\]
Thus $v_0\in (S_1\cap\cdots\cap S_r)\setminus S$.

If $v\in S$, then $L\not\subseteq  F_v$, so $a_i^{1/p}\not\in F_v$ for some $1\leq i\leq r$.
This shows that $S\subseteq S_1\cup\cdots\cup S_r$.
Hence we may decompose $S=S'_1\discup\cdots\discup S'_r$ with $S'_i\subseteq S_i$, $i=1,2\nek r$.
Note that $v_0\not\in S'_i$ for every $i$.

For every $i$ let
\[
t_i:=\sum_{v\in S'_i} \Inv_v(\alp_{F_v}).
\]
Then (\ref{LGP})  gives rise to $\alp_i\in H^2(F)$ with local invariants
\[
\Inv_v((\alp_i)_{F_v})
=\begin{cases}
\Inv_v(\alp_{F_v}), & \hbox{if } v\in S'_i,\\
-t_i,& \hbox{if } v=v_0,\\
0,& \hbox{otherwise}.\\
\end{cases}
\]

{\bf Claim 1:} \ For every place $w$ on $F$,
\begin{equation}
\label{equality of local invariants}
\Inv_w(\alp_{F_w})=\sum_{i=1}^r\Inv_w((\alp_i)_{F_w}).
\end{equation}
Indeed, when $w\in S'_i$ for some $i$, this follows from the disjointness of $S'_1\nek S'_r$.
When $w=v_0$ we compute using (\ref{LGP}):
\[
\begin{split}
\Inv_w(\alp_{F_w})=0&=-\sum_v\Inv_v(\alp_{F_v})=-\sum_{v\in S}\Inv_v(\alp_{F_v})\\
&=-\sum_{i=1}^r\sum_{v\in S'_i} \Inv_v(\alp_{F_v})=-\sum_{i=1}^rt_i=\sum_{i=1}^r\Inv_w((\alp_i)_{F_w}).
\end{split}
\]
For all other places, the left-hand side of (\ref{equality of local invariants}) is zero, by the definition of $S$, and all summands on the right-hand side are zero by the choice of $\alp_i$.
This proves the claim.

We conclude from Claim 1 and from (\ref{LGP})  that $\alp=\sum_{i=1}^r\alp_i$ in $H^2(F)$.

\medskip

{\bf Claim 2:} \
For every $1\leq i\leq r$ and every place $u$ on $F(a_i^{1/p})$ one has
\[
(\alp_i)_{F(a_i^{1/p})_u}=0.
\]
To see this let $v$ be the place on $F$ which lies under $u$.

\medskip

Case 1: $v\in S_i$.
Then $p=[F_v(a_i^{1/p}):F_v]\bigm|[F(a_i^{1/p})_u:F_v]$, so as we have seen, $\res\colon H^2(F_v)\to H^2(F(a_i^{1/p})_u)$ is the zero map.
In particular, $(\alp_i)_{F(a_i^{1/p})_u}=0$.

Case 2: $v\not\in S_i$.
Then $a_i^{1/p}\in F_v$ and $v\neq v_0$.
Hence $F(a_i^{1/p})_u=F_v$.
The choice of $\alp_i$ implies that $\Inv_v((\alp_i)_{F_v})=0$, so again,
$(\alp_i)_{F(a_i^{1/p})_u}=(\alp_i)_{F_v}=0$.

\medskip

We conclude from Claim 2 and from (\ref{LGP}) that $(\alp_i)_{F(a_i^{1/p})}=0$, $i=1,2\nek r$.
Therefore $\alp_i\in (a_i)\cup H^1(F)$ (Example \ref{example r=1}).
It follows that
\[
\alp=\sum_{i=1}^r\alp_i\in (a_1)\cup H^1(F)+\cdots+(a_r)\cup H^1(F).
\qedhere
\]
\end{proof}

\section{The case $p=2$}
Let $p=2$ and let $F$ be a field of characteristic $\neq2$.
By a classical result of Albert \cite{Albert39}, every central simple $F$-algebra of degree $4$ and exponent $2$ is $F$-isomorphic to a tensor product of two $F$-quaternion algebras.
This was extended by Tignol \cite{Tignol79}*{Th.\ 1} as follows.
Recall that an involution on a central simple $F$-algebra is of the \textsl{first kind} if it is the identity on $F$.

\begin{thm}[Tignol]
\label{stronger Albert}
Let $A$ be a central simple $F$-algebra with involution of the first kind and which is split by
a Galois extension $M$ of $F$ with Galois group $(\dbZ/2\dbZ)^2$.
Let $L_1,L_2$ be quadratic extensions of $F$ with $M=L_1L_2$.
Then there are quaternion $F$-algebras $Q_1,Q_2$ such that $L_i\subset Q_i\subset A$, $i=1,2$, and $A\isom_F Q_1\tensor_F Q_2$.
\end{thm}

For a closely related result see \cite{Rowen84}*{Cor.\ 5}.

The case $p=2$, $r=2$ of the Main Theorem is an easy corollary of Theorem \ref{stronger Albert}.
In a different terminology it was obtained by Tignol in \cite{Tignol81a}*{Cor.\ 2.8}, however we provide a short proof showing its relation to Theorem \ref{stronger Albert}.

\begin{proof}[Proof of Case (2) of the Main Theorem]
Let $a_1,a_2\in F^\times$ and denote $M=F(\sqrt a_1,\sqrt a_2)$.
Thus in the terminology of sequence (\ref{cup product-restriction}), $K=G_M$.
Since the sequence (\ref{cup product-restriction}) is always a complex, we need to show that for every central simple $F$-algebra $A$ of exponent $2$ and which splits in $M$
the similarity class $[A]$ of $A$ in ${}_2\Br(F)$ is contained in $(a_1)\cup H^1(G)+(a_2)\cup H^1(G)$.
We may assume that $A$ does not split in $F$.
If $a_1,a_2$ have $\dbF_2$-linearly dependent cosets in $F^\times/(F^\times)^2$, then we are done by Example \ref{example r=1}.
So assume that $a_1,a_2$  have $\dbF_2$-linearly independent cosets in $F^\times/(F^\times)^2$.

Since $A$ splits in $M$, it is similar to a central simple $F$-algebra $A'$ of degree $[M:F]=4$  and which contains $M$ \cite{Draxl83}*{p.\ 64, Th.\ 7}.
The exponent of $A'$ is also $2$, and therefore it has an involution of the first kind \cite{Albert39}*{Ch.\ X, Th.\ 19}.
Let $L_i=F(\sqrt{a_i})$, $i=1,2$.
Theorem \ref{stronger Albert} yields quaternion $F$-subalgebras $Q_1,Q_2$ of $A'$ which contain $L_1,L_2$ respectively, and such that $A'\isom_F Q_1\tensor_F Q_2$.
Then $Q_1\isom_F(a_1,x)_F$ and $Q_2\isom_F(a_2,y)_F$ for some $x,y\in F^\times$  \cite{Draxl83}*{p.\ 104, Th.\ 4}.
Therefore $[A]=[A']=(a_1)\cup (x)+(a_2)\cup(y)$, as desired.
\end{proof}

\begin{rem}
\label{rems on the odd prime case}
\rm
There are no known direct generalizations of Theorem \ref{stronger Albert} and \cite{Rowen84}*{Cor.\ 5} for odd primes.
For instance, when $p=3$, it seems that the best result to date is that a central simple $F$-algebra which contains a maximal subfield $L$
which is Galois over $F$ with $\Gal(L/F)\isom\dbZ/3\times\dbZ/3$ (i.e., $A$ is a $\dbZ/3\times\dbZ/3$-crossed product over $F$) is similar to the tensor product of $\leq31$ symbol algebras of degree $3$ over $F$ \cite{Matzri14}.
\end{rem}

\section{Massey products}
\label{section on Massey products}
We recall the definition and basic properties of Massey products of degree $1$ cohomology elements.
For more information see e.g., Fenn \cite{Fenn83},  Kraines \cite{Kraines66},  Dwyer \cite{Dwyer75} (and in a more general setting, May \cite{May69}).
Note  that the various sources use different sign conventions.

We first recall that a \textsl{differential graded algebra}  over a ring $R$ ($R$-DGA)  is a graded $R$-algebra
$C^\bullet=\bigoplus_{s=0}^\infty C^s$ equipped with $R$-module homomorphisms
$\partial^s\colon C^s\to C^{s+1}$ such that $(C^\bullet,\bigoplus_{s=0}^\infty\partial^s)$ is a complex satisfying the \textsl{Leibnitz rule}
$\partial^{r+s}(ab)=\partial^r(a)b+(-1)^ra\partial^s(b)$ for  $a\in C^r$, $b\in C^s$.
Set $Z^r=\Ker(\partial^r)$, $B^r=\Img(\partial^{r-1})$, and $H^r=Z^r/B^r$, and let $[c]$ denote the class of $c\in Z^r$ in $H^r$.

We fix  an integer $n\geq 2$.
Consider a system $c_{ij}\in C^1$, where $1\leq i\leq j\leq n$ and $(i,j)\neq(1,n)$.
For any $i,j$ satisfying $1\leq i\leq j\leq n$ (including $(i,j)=(1,n)$) we define
\[
\widetilde c_{ij}=-\sum_{r=i}^{j-1}c_{ir}c_{r+1,j}\in C^2.
\]
One says that $(c_{ij})$ is a  \textsl{defining system of size $n$} in $C^\bullet$ if
$\partial c_{ij}=\widetilde c_{ij}$ for every $1\leq i\leq j\leq n$ with $(i,j)\neq(1,n)$.
We also say that the defining system $(c_{ij})$ is \textsl{on $c_{11}\nek c_{nn}$}.
Note that then $c_{ii}$ is a $1$-cocycle, $i=1,2\nek n$.
Further, $\widetilde c_{1n}$ is a $2$-cocycle (\cite{Kraines66}*{p.\ 432}, \cite{Fenn83}*{p.\ 233}).
Its cohomology class depends only on the cohomology classes $[c_{11}]\nek[c_{nn}]$ \cite{Kraines66}*{Th.\ 3}.
Given $c_1\nek c_n\in Z^1$, the \textsl{$n$-fold Massey product} of $\langle [c_1]\nek [c_n]\rangle$ is the subset of $H^2$ consisting of all cohomology classes $[\widetilde{c_{1n}}]$
obtained from defining systems $(c_{ij})$ of size $n$ on $c_1\nek c_n$ in $C^\bullet$.
This construction is functorial in the natural sense.
When this subset is nonempty one says that $\langle [c_1]\nek [c_n]\rangle$ is \textsl{defined}.
Note that $\langle [c_1]\nek [c_n]\rangle$ contains $0$ if and only if there is an array $(c_{ij})$, $1\leq i\leq j\leq n$, in $C^1$ such that $\partial c_{ij}=\widetilde c_{ij}$ for every $1\leq i\leq j\leq n$
(including $(i,j)=(1,n)$).
In this case one says that the Massey product $\langle [c_1]\nek [c_n]\rangle$ is \textsl{trivial}.

When $n=2$, $\langle [c_1],[c_2]\rangle$ is always defined and consists only of $-[c_1][c_2]$.

Next we record some well-known facts on the case $n=3$.

\begin{prop}
\label{structure of triple Massey products}
Let $c_1,c_2,c_3\in Z^1$.
\begin{enumerate}
\item[(a)]
$\langle [c_1],[c_2],[c_3]\rangle$ is defined if and only if $[c_1][c_2]=[c_2][c_3]=0$;
\item[(b)]
If $(c_{ij})$ is a defining system on $[c_1],[c_2],[c_3]$, then
$\langle [c_1],[c_2],[c_3]\rangle=[\widetilde {c_{13}}]+[c_1]H^1+[c_3]H^1$.
\end{enumerate}
\end{prop}
\begin{proof}
(a) \quad
Having a defining system on $c_1,c_2,c_3$ means that there exist $c_{12},c_{23}\in C^1$ with $\partial c_{12}=-c_1c_2$ and $\partial c_{23}=-c_2c_3$,
i.e., $c_1c_2,c_2c_3\in B^2$.

\medskip

(b) \quad
Suppose that $(c'_{ij})$ is another defining system on $c_1,c_2,c_3$.
For $d_{12}=c'_{12}-c_{12}$ and $d_{23}=c_{23}-c'_{23}$ we have $\partial d_{12}=\widetilde{c'_{12}}-\widetilde{c_{12}}=-c_1c_2+c_1c_2=0$.
Thus $d_{12}\in Z^1$, and similarly $d_{23}\in Z^1$.
By a direct calculation $[\widetilde{c'_{13}}]=[\widetilde{c_{13}}]+[c_1][d_{23}]+
[c_3][d_{12}]$.

Conversely, for every $d_{12},d_{23}\in Z^1$, the system $(c'_{ij})$ is also a defining system on $c_1,c_2,c_3$, where we take
$c'_{ii}=c_i$, $c'_{12}=c_{12}-d_{12}$ and $c'_{23}=c_{23}+d_{23}$.
One has $[\widetilde{c'_{13}}]=[\widetilde{c_{13}}]+[c_1][d_{23}]+[c_3][d_{12}]$.
\end{proof}

Now for the fixed prime number $p$, let $G$ be a profinite group acting trivially on $\dbZ/p$.
Let $C^\bullet=\bigoplus_{s=0}^\infty C^s(G,\dbZ/p)$ be the $\dbZ/p$-DGA of continuous cochains from $G$ to $\dbZ/p$, with the cup product.
Thus in our previous notation, $H^i=H^i(G)$.

\begin{prop}
\label{cohomological lemma}
Let $\chi_1,\chi_2,\chi_3\in H^1(G)$.
Suppose that the triple Massey product $\langle\chi_1,\chi_2,\chi_3\rangle$ is defined, and that the cup product--restriction property holds for $\chi_1,\chi_3$.
Then $0\in\langle\chi_1,\chi_2,\chi_3\rangle$.
\end{prop}
\begin{proof}
Take $\alp\in\langle\chi_1,\chi_2,\chi_3\rangle$.
Let $K=\Ker(\chi_1)\cap\Ker(\chi_3)$.
The functoriality of the Massey product implies that
\[
\res_K(\alp)\in \langle\res_K(\chi_1),\res_K(\chi_2),\res_K(\chi_3)\rangle=\langle0,\res_K(\chi_2),0\rangle=\{0\}.
\]
By the exactness of (\ref{cup product-restriction}), $\alp\in \Img(\Lam_{(\chi_1,\chi_3)})$, that is,
 $\alp=\chi_1\cup\beta_1+\chi_3\cup\beta_3$ for some $\beta_1,\beta_3\in H^1(G)$.
Now Proposition \ref{structure of triple Massey products} implies that
\[
0=\alp-\chi_1\cup\beta_1-\chi_3\cup\beta_3\in \langle\chi_1,\chi_2,\chi_3\rangle.
\qedhere
\]
\end{proof}

Dwyer \cite{Dwyer75} relates $n$-fold Massey products in $C^\bullet(G,\dbZ/p)$
to unipotent upper-triangular $n+1$-dimensional representations of $G$ as follows
(\cite{Dwyer75} works in a discrete context and with more a general coefficients ring;
see \cite{Efrat14}*{\S8} for the profinite context, and \cite{Wickelgren12} for a generalization to the case of non-trivial actions).
For $n\geq2$ let $\dbU_{n+1}(\dbF_p)$ be as before the group of all unipotent upper-triangular $(n+1)\times(n+1)$-matrices over $\dbF_p$.
Its center consists of all matrices which are $0$ on all off-diagonal entries, except possibly for entry $(1,n+1)$.
Let $\bar\dbU_{n+1}(\dbF_p)$ be the quotient of $\dbU_{n+1}(\dbF_p)$ by this center.
Its elements may be viewed as  unipotent upper-triangular $(n+1)\times(n+1)$-matrices with the $(1,n+1)$-entry deleted.
We notice that $\dbU_{n+1}(\dbF_p)$ is a $p$-group, and its Frattini subgroup is the kernel of the epimorphism
$\dbU_{n+1}(\dbF_p)\to(\dbZ/p)^n$, $(c_{ij})\mapsto(c_{12},c_{23}\nek c_{n,n+1})$, and similarly for $\bar \dbU_{n+1}(\dbF_p)$.

Given an array $(c_{ij})$, $1\leq i\leq j\leq n$, in $C^1(G,\dbZ/p)$, we define a continuous map $\gam\colon G\to\dbU_{n+1}(\dbF_p)$
by $\gam(\sig)_{ij}=(-1)^{j-i}c_{i,j-1}(\sig)$ for $\sig\in G$ and for $1\leq i<j\leq n+1$.
Then $\widetilde{c_{ij}}=\partial c_{ij}$ for every $i<j$ if and only if $\gam\colon G\to\dbU_{n+1}(\dbF_p)$ is a homomorphism.
Similarly, $\widetilde{c_{ij}}=\partial c_{ij}$ for every $i<j$ with $(i,j)\neq(1,n)$ if and only if the induced map $\bar\gam\colon G\to\bar\dbU_{n+1}(\dbF_p)$ is a homomorphism.
We write $\gam_{ij},\bar\gam_{ij}$ for the projections of $\gam,\bar\gam$, respectively, on the $(i,j)$-coordinate.

\begin{prop}
\label{Massey products and homorphisms}
Let $\chi_1\nek\chi_n\in H^1(G)$ be $\dbF_p$-linearly independent.
\begin{enumerate}
\item[(b)]
$\langle\chi_1\nek\chi_n\rangle$ is defined if and only if
there exists a continuous homomorphism $\bar\gam\colon G\to\bar\dbU_{n+1}(\dbF_p)$ such that $\bar\gam_{i,i+1}=\chi_i$, $i=1,2\nek n$,
\item[(B)]
$\langle\chi_1\nek\chi_n\rangle$ is trivial if and only if
there exists a continuous homomorphism $\gam\colon G\to\dbU_{n+1}(\dbF_p)$ such that $\gam_{i,i+1}=\chi_i$, $i=1,2\nek n$.
\end{enumerate}
Moreover, such homomorphisms $\gam,\bar\gam$ are necessarily surjective.
\end{prop}
\begin{proof}
(a) and (b) follow from the above discussion.
The surjectivity of $\gam$ and $\bar\gam$ follows by a Frattini argument.
\end{proof}

This and Proposition \ref{structure of triple Massey products}(a) imply the following facts, mentioned in the Introduction, which are also implicit in \cite{MinacTan14}*{Cor.\ 3.2}:

\begin{cor}
\label{Massey products and Un}
Let $\chi_1,\chi_2,\chi_3\in H^1(G)$ be $\dbF_p$-linearly independent.
Then:
\begin{enumerate}
\item[(a)]
$\langle\chi_1,\chi_2,\chi_3\rangle$ is defined if and only if there exist continuous epimorphisms $\gam',\gam''\colon G\to\dbU_3(\dbF_p)$
such that
$\gam'_{12}=\chi_1$, $\gam'_{23}=\chi_2$, $\gam''_{12}=\chi_2$, $\gam''_{23}=\chi_3$.
\item[(b)]
 $\langle\chi_1,\chi_2,\chi_3\rangle$ is trivial if and only if there exists a continuous epimorphism $\gam\colon G\to\dbU_4(\dbF_p)$ such that
$\gam_{12}=\chi_1$, $\gam_{23}=\chi_2$ and $\gam_{34}=\chi_3$.
\end{enumerate}
\end{cor}


\begin{bibdiv}
\begin{biblist}


\bib{Albert39}{book}{
   author={Albert, A. Adrian},
   title={Structure of Algebras},
   series={American Mathematical Society Colloquium   Publications, Vol. XXIV},
   publisher={American Mathematical Society, Providence, R.I.},
   date={1939},
}

\bib{Arason75}{article}{
   author={Arason, J{\'o}n Kr.},
   title={Cohomologische Invarianten quadratischer Formen},
   journal={J. Algebra},
   volume={36},
   date={1975},
   pages={448--491},
 }


\bib{ArtinTate}{book}{
   author={Artin, Emil},
   author={Tate, John},
   title={Class Field Theory},
   note={Reprinted with corrections from the 1967 original},
   publisher={AMS Chelsea Publishing, Providence, RI},
   date={2009},
   pages={viii+194},
}

\bib{Draxl83}{book}{
author={Draxl, P.K.},
title={Skew Fields},
series={London Math. Soc.\ Lect.\ Notes Series},
volume={81},
publisher={Cambridge University Press},
place={Cambridge},
date={1983},
}

\bib{Dwyer75}{article}{
   author={Dwyer, William G.},
   title={Homology, Massey products and maps between groups},
   journal={J. Pure Appl. Algebra},
   volume={6},
   date={1975},
   pages={177\ndash190},
}

\bib{Efrat14}{article}{
   author={Efrat, Ido},
   title={The Zassenhaus filtration, Massey products, and representations of profinite groups},
   journal={Adv.\ Math.},
   volume={263},
   date={2014},
   pages={389\ndash411},
}

\bib{EfratMinac11}{article}{
author={Efrat, Ido},
author={Min\' a\v c, J\'an},
title={On the descending central sequence of absolute Galois groups},
journal={Amer.\ J.\ Math.},
volume={133},
date={2011},
pages={1503\ndash1532},
}

\bib{ElmanLamTignolWadsworth83}{article}{
author={Elman, Richard},
author={Lam, T. Y.},
author={Tignol, Jean-Pierre},
author={Wadsworth, A.R.},
title={Witt rings and Brauer groups under multiquadratic extensions, I},
journal={Amer. J. Math.},
volume={105},
date={1983},
pages={1119\ndash1170},
}


\bib{Fenn83}{book}{
author={Fenn, Roger A.},
title={Techniques of Geometric Topology},
Series={London Math.\ Soc.\ Lect. Notes Series},
volume={57},
publisher={Cambridge Univ. Press},
date={1983},
place={Cambridge}
}

\bib{HopkinsWickelgren15}{article}{
   author={Hopkins, Michael J.},
   author={Wickelgren, Kirsten G.},
   title={Splitting varieties for triple Massey products},
   journal={J. Pure Appl. Algebra},
   volume={219},
   date={2015},
   pages={1304--1319},
}


\bib{Kraines66}{article}{
author={Kraines, David},
title={Massey higher products},
journal={Trans.\ Amer.\ Math.\ Soc.},
volume={124},
date={1966},
pages={431\ndash449},
}

\bib{LemireMinacSwallow07}{article}{
author={Lemire, Nicole},
author={Min\'a\v c, J\'an},
author={Swallow, John},
title={Galois module structure of Galois Cohomology and partial Euler-Poincare Characteristics},
journal={J.\ reine angew.\ Math.},
volume={613},
date={2007},
pages={147\ndash173},
}

\bib{Matzri14}{article}{
author={Matzri, Eliyahu},
title={$\dbZ_3\times\dbZ_3$-crossed products},
journal={J.\ Algebra},
volume={418},
date={2014},
pages={1\ndash7},
}



\bib{May69}{article}{
author={May, J. Peter},
title={Matric Massey products},
journal={J. Algebra},
volume={12},
pages={533\ndash 568},
date={1969}
}

\bib{McKinnie11}{article}{
   author={McKinnie, Kelly},
   title={Degeneracy and decomposability in abelian crossed products},
   journal={J. Algebra},
   volume={328},
   date={2011},
   pages={443\ndash460},
}

\bib{MinacTan13}{article}{
author={Min\'a\v c, J\'an},
author={T\^an, Nguyen Duy},
title={Triple Massey products and Galois theory},
journal={J.\ Eur.\ Math.\ Soc.},
status={to appear},
eprint={arXiv:1307.6624},
date={2013},
}

\bib{MinacTan14}{article}{
author={Min\'a\v c, J\'an},
author={T\^an, Nguyen Duy},
title={Triple Massey products over global fields},
eprint={arXiv:1407.4586},
date={2014},
}

\bib{MinacTan15}{article}{
   author={Min\'a\v c, J\'an},
   author={T\^an, Nguyn Duy},
   title={The kernel unipotent conjecture and the vanishing of Massey    products for odd rigid fields},
   journal={Adv. Math.},
   volume={273},
   date={2015},
   pages={242--270},
   status={(with an appendix by I.\ Efrat, J.\ Min\'a\v c, and N.D. T\^an)},
}



		
\bib{Rowen82}{article}{
   author={Rowen, Louis Halle},
   title={Cyclic division algebras},
   journal={Israel J. Math.},
   volume={41},
   date={1982},
   pages={213--234},
   note={Correction: Israel J.\ Math.\ {\bf43} (1982),  277\ndash 280},
}
		


\bib{Rowen84}{article}{
   author={Rowen, Louis H.},
   title={Division algebras of exponent $2$ and characteristic $2$},
   journal={J. Algebra},
   volume={90},
   date={1984},
   pages={71--83},
}

\bib{Rowen}{book}{
   author={Rowen, Louis Halle},
   title={Graduate Algebra: Noncommutative View},
   series={Graduate Studies in Mathematics},
   volume={91},
   publisher={Amer.\ Math.\ Soc., Providence, RI},
   date={2008},
   pages={xxvi+648},
}

\bib{Saltman79}{article}{
   author={Saltman, David J.},
   title={Indecomposable division algebras},
   journal={Comm. Algebra},
   volume={7},
   date={1979},
   pages={791--817},
}

\bib{SerreLocalFields}{book}{
   author={Serre, Jean-Pierre},
   title={Local Fields},
   series={Grad.\ Texts Math.},
   volume={67},
   publisher={Springer-Verlag, New York-Berlin},
   date={1979},
   pages={viii+241},
}

\bib{Tignol79}{article}{
   author={Tignol, J.-P.},
   title={Central simple algebras with involution},
   conference={
      title={Ring theory (Proc. Antwerp Conf.},
      address={NATO Adv. Study Inst.), Univ. Antwerp, Antwerp},
      date={1978},
   },
   book={
      series={Lecture Notes in Pure and Appl. Math.},
      volume={51},
      publisher={Dekker, New York},
   },
   date={1979},
   pages={279--285},
}

\bib{Tignol81a}{article}{
author={Tignol, Jean-Pierre},
title={Corps \`a involution neutralis\'es par une extension ab\'elienne \'el\'ementaire,},
conference={
    title={Groupe de Brauer (Les Plans-sur-Bex 1980, M.\ Kervaire and M.\ Ojanguren, Eds.)}, },
book={series={Lecture Notes in Math.},
    volume={844},
    publisher={Springer, Berlin},
    }
date={1981},
pages={1\ndash34},
}



\bib{Tignol81b}{article}{
author={Tignol, Jean-Pierre},
 title={Produits crois\'es ab\'eliens},
 journal={J. Algebra},
 volume={70},
 date={1981},
  pages={420--436},
}





\bib{Tignol87}{article}{
author={Tignol, Jean-Pierre},
title={Alg\`ebres ind\'ecomposables d'exposant premier},
journal={Adv.\ Math.},
volume={65},
pages={205\ndash228},
date={1987},
}



\bib{Voevodsky03}{article}{
author={Voevodsky, Vladimir},
title={Motivic cohomology with $\dbZ/2$-coefficients},
journal={Publ. Math. Inst. Hautes \'Etudes Sci.},
volume={98},
pages={59\ndash104},
date={2003},
}

\bib{Wickelgren12}{article}{
   author={Wickelgren, Kirsten},
   title={$n$-nilpotent obstructions to $\pi_1$ sections of $\mathbb{P}^1-\{0,1,\infty\}$ and Massey products},
   conference={
      title={Galois-Teichm\"uller theory and arithmetic geometry},
   },
   book={
      series={Adv. Stud. Pure Math.},
      volume={63},
      publisher={Math. Soc. Japan, Tokyo},
   },
   date={2012},
   pages={579--600},
  }

\end{biblist}
\end{bibdiv}
\end{document}